\documentclass[a4paper,11pt,leqno]{amsart}
\usepackage{a4wide}
\usepackage{amsmath}
\usepackage[italian, english]{babel}
\usepackage{amsfonts}
\usepackage{amssymb}
\usepackage{dsfont}
\usepackage{mathrsfs}
\usepackage{graphicx}
\usepackage{fancyhdr}
\usepackage{amsthm}
\usepackage{empheq}
\usepackage{cases}
\usepackage[all]{xy}
\usepackage{stmaryrd}
\usepackage[colorlinks, citecolor=blue, linkcolor=red]{hyperref}
\usepackage{mathrsfs}
\def\SS{{{\mathbb S}}}

\def\RR{{\mathbb R}}

\newcommand{\erre}{\mathds{R}}

\newcommand{\ricc}{\operatorname{Ric}}

\newcommand{\pa}[1]{{\left(#1\right)}}                  
\newcommand{\sq}[1]{{\left[#1\right]}}                  
\newcommand{\abs}[1]{{\left|#1\right|}}                 

\renewcommand{\tilde}[1]{\widetilde{#1}}

\newtheorem{theorem}{\textbf{Theorem}}[section]
\newtheorem{lemma}[theorem]{\textbf{Lemma}}
\newtheorem{proposition}[theorem]{\textbf{Proposition}}

\theoremstyle{remark}
\newtheorem{rem}[theorem]{\textbf{Remark}}

\numberwithin{equation}{section}
\title[GRS with vanishing conditions on Weyl]
{Gradient Ricci solitons \\ with vanishing conditions on Weyl}

\date{\today} 
\keywords{Ricci solitons, rigidity results, Weyl tensor}

\subjclass[2010]{53C20, 53C25.}

\begin{document}

\maketitle

\date{\today}

\begin{center}
\textsc{\textmd{G. Catino\footnote{Politecnico di Milano, Italy.
Email: giovanni.catino@polimi.it.}, P.
Mastrolia\footnote{Universit\`{a} degli Studi di Milano, Italy.
Email: paolo.mastrolia@gmail.com.}, D. D. Monticelli\footnote{Politecnico di Milano, Italy. Email: dario.monticelli@polimi.it. \\The three authors are supported
by GNAMPA project ``Analisi Globale, PDE's e Strutture Solitoniche''.}}}
\end{center}

\begin{abstract}
We classify complete gradient Ricci solitons satisfying a fourth-order vanishing condition on the Weyl tensor, improving previously known results. More precisely, we show that any $n$-dimensional ($n\geq 4$) gradient shrinking Ricci soliton with fourth order divergence-free Weyl tensor is either Einstein, or a finite quotient of $N^{n-k}\times \mathbb{R}^k$, $(k > 0)$, the product of a Einstein manifold $N^{n-k}$ with the Gaussian shrinking soliton $\mathbb{R}^k$. The technique applies also to the steady and expanding cases in all dimensions. In particular, we prove that a three dimensional gradient steady soliton with third order divergence-free Cotton tensor, i.e. with vanishing double divergence of the Bach tensor, is either flat or isometric to the Bryant soliton. 
\end{abstract}


\section{Introduction}

A Riemannian manifold $(M^n ,g)$ of dimension $n\geq 3$ is a {\em gradient Ricci soliton} if there exists a smooth function $f$ on $M$ such that
$$
\ricc + \nabla^2 f \, = \, \lambda \, g
$$
for some constant $\lambda$. The Ricci soliton is called {\em shrinking} if $\lambda>0$, {\em steady} if $\lambda=0$ and {\em expanding} if $\lambda<0$. When the potential function $f$ is a constant a gradient Ricci soliton is simply an Einstein manifold. Ricci solitons generate self-similar solutions of the Ricci flow, play a fundamental role in the formation of singularities of the flow and have been studied by several authors (see H.-D. Cao~\cite{cao1, cao2} for nice overviews). As it is clear from the definition, the Ricci solitons equation can be interpreted as a prescribing condition on the Ricci tensor of $g$, that is on the {\em trace part} of the Riemann tensor (see for instance the interesting paper \cite{bou}). Thus, we can expect classification results for these structures only assuming further conditions on the {\em traceless part} of the Riemann tensor, i.e. on the {\em Weyl tensor} $W$, if $n\geq 4$.

Three dimensional complete gradient {\em shrinking} Ricci solitons are classified and indeed it is well known that they are finite quotients of either the round sphere $\SS^3$, or the Gaussian shrinking soliton $\RR^3$, or the round cylinder $\SS^2\times \RR$ (see Ivey~\cite{ivey1} for the compact case and Perelman~\cite{per1}, Ni--Wallach~\cite{niwal} and H.-D.Cao--B.-L.Chen--X.-P.Zhu~\cite{caochenzhu} for the complete case). 

In higher dimensions, classification results for Ricci shrinkers have been obtained by several authors under curvature conditions on the Weyl tensor. Z.-H.Zhang \cite{zhang}, based on the work of Ni-Wallach \cite{niwal}, showed that complete locally conformally flat gradient shrinking Ricci solitons, i.e. with
$$
W_{ikjl}=0\,,
$$ 
are isometric to finite quotients of either $\SS^n$, $\RR^n$, or $\SS^{n-1}\times\RR$ (see also the works of Eminenti--La Nave--Mantegazza \cite{emilanman}, Petersen--Wylie \cite{petwyl}, X.Cao--B.Wang--Z.Zhang \cite{caowanzha}). Other rigidity results have been obtained under suitable pointwise or integral pinching conditions on the Weyl tensor by Catino \cite{cat1, cat2} and X.Cao--Tran \cite{caotra}. In dimension four, X.Chen-Y.Wang \cite{chewan} (see also H.-D.Cao-Q.Chen \cite{caoche2}) proved that half conformally flat (i.e. $W^{\pm}=0$) gradient shrinking Ricci solitons are finite quotients of $\SS^4$, $\mathbb{CP}^2$, $\RR^4$, or $\SS^3\times\RR$. 

Under the weaker condition of harmonic Weyl tensor, i.e. 
$$
\operatorname{div}(W)=\nabla_l W_{ikjl}=0\,,
$$
Fernandez-Lopez--Garcia-Rio \cite{fergar} and Munteanu--Sesum \cite{munses} proved that $n$-dimensional complete gradient shrinking solitons are either Einstein, or finite quotients of $N^{n-k}\times \mathbb{R}^k$, $(k > 0)$, the product of a Einstein manifold $N^{n-k}$ with the Gaussian shrinking soliton $\mathbb{R}^k$. In the case $n=4$, a stronger result have been obtained by J.-Y.Wu--P.Wu--Wylie \cite{wuwuwyl}, assuming the Weyl tensor is half harmonic (i.e. $\operatorname{div}(W^{\pm})=0$).

It is interesting to observe that the aforementioned results can be interpreted as rigidity results under {\em zero and first order} vanishing conditions on the Weyl tensor. 

Recently, H.-D.Cao--Q.Chen \cite{caoche2} showed that Bach-flat gradient shrinking Ricci solitons, i.e. with
$$
B_{ij} =  \frac{1}{n-3}\nabla_k \nabla_l W_{ikjl} + \frac{1}{n-2}R_{kl}W_{ikjl}=0 \,,
$$
are either Einstein, or finite quotients of $\RR^n$ or $N^{n-1}\times\RR$, where $N^{n-1}$ is an $(n-1)$-dimensional Einstein manifold. In the same spirit as before, this can be seen as a vanishing condition involving {\em second and zero order} terms in Weyl, which a posteriori captures a more rigid class of solitons than in the harmonic Weyl case.

Gradient {\em steady} Ricci solitons are less rigid, but many results in the same spirit have been obtained. It is well-known that compact gradient steady solitons must be Ricci flat. In dimension $n=2$, the only gradient steady Ricci soliton with positive curvature is the Hamilton's cigar $\Sigma^2$, see Hamilton \cite{ham1}. In dimension three, the classification of complete gradient steady Ricci solitons is still open. Known examples are given by quotients of $\RR^{3}$, $\Sigma^{2}\times\RR$ and the rotationally symmetric one constructed by Bryant \cite{bry}. In the paper by Brendle \cite{bre}, it was shown that the Bryant soliton is the only nonflat, $k$-noncollapsed, steady soliton, proving a famous conjecture by Perelman \cite{per1}. Other results in the steady three dimensional case have been obtained in H.-D.Cao--Catino--Q.Chen--Mantegazza--Mazzieri \cite{caocatchemanmaz} and Catino-Mastrolia-Monticelli \cite{catmasmon}. In particular, in \cite{caocatchemanmaz} the author showed rigidity just assuming that the Bach tensor is divergence free, which is equivalent to a second order vanishing condition on the Cotton tensor (see Section \ref{sec_def}).

In higher dimensions, H.-D.Cao--Q.Chen proved in \cite{caoche1} 
that complete $n$-dimensional ($n\ge 3$) locally conformally flat gradient steady Ricci solitons are isometric to either a finite quotient of $\RR^n$ or the Bryant soliton. The same result for $n\ge 4$ was proved independently by Catino--Mantegazza in \cite{catman} by using different methods. When $n=4$, X.Chen--Y.Wang \cite{chewan} showed that any four dimensional complete half-conformally flat gradient steady Ricci soliton is either Ricci flat, or isometric to the Bryant Soliton. Again, these are rigidity results under zero order conditions on Weyl. 

Classification results have been obtained in \cite{caocatchemanmaz} for Bach flat steady solitons case in dimension $n\geq 4$. In particular, it follows that Bach flatness implies local conformal flatness. It is still an open question if similar results can be obtained under first order vanishing conditions on Weyl.

The case of {\em expanding} solitons is clearly the less rigid. However, several interesting results under vanishing conditions on Weyl have been obtained, see for instance \cite{caocatchemanmaz, catman}.

\

We remark that all the aforementioned results rely on {\em vanishing conditions} involving zero, first or specific second order derivatives of the Weyl tensor. The aim of this paper is to obtain classification of gradient Ricci solitons under much weaker assumptions, only requiring the vanishing of a {\em fourth order} divergence of Weyl. Since, in dimension three, Weyl is identically null, our results in this case will require the vanishing of a third order divergence of the Cotton tensor. Hence, in every dimension, we only require a {\em scalar condition} on the Weyl (Cotton) curvature of the soliton in order to conclude.

In their study of the geometry of locally conformally flat and Bach flat gradient solitons, H.-D. Cao and Q. Chen \cite{caoche1, caoche2} introduced a three tensor $D$ related to the geometry of the level surfaces of the potential $f$. The vanishing of $D$, which is a consequence of the curvature assumption on Weyl, is a crucial ingredient in their classification results. In particular, they showed that every gradient Ricci soliton satisfies the two conditions
 \begin{align*}
    &C_{ijk}+f_t W_{tijk} = D_{ijk},\\     &B_{ij} = \frac{1}{n-2}\sq{\nabla_k D_{ijk}+\pa{\frac{n-3}{n-2}}f_tC_{jit}}.
  \end{align*}
  The above equations can be intended as integrability conditions for gradient Ricci solitons. In a recent work by Catino--Mastrolia--Monticelli--Rigoli \cite{catmasmonrig} the authors derived higher order integrability conditions involving the tensor $D$ and the Cotton tensor, which will play an important role in our proof. For the sake of completeness, we will recall and prove them in Section \ref{sec_int}.

In order to precisely state our results, we introduce the following definitions
\begin{eqnarray*}
\operatorname{div}^4(W) &=& \nabla_k\nabla_j\nabla_l\nabla_i W_{ikjl} \\
\operatorname{div}^3(C) &=& \nabla_i\nabla_j\nabla_kC_{ijk} \,,
\end{eqnarray*}
where $W$ and $C$ are the Weyl and the Cotton tensors, respectively (see Section \ref{sec_def} for definitions and notation). Note that, in dimension $n\geq 4$, $\operatorname{div}^4(W)=0$ if and only if $\operatorname{div}^3(C)=0$ (see equation \eqref{def_Cotton_comp_Weyl}).

Our first main result is the following classification theorem for gradient shrinking Ricci solitons of dimension $n\geq 4$ with $\operatorname{div}^4(W)=0$.

\begin{theorem}\label{thm_shr} Every complete gradient shrinking Ricci soliton of dimension $n\geq 4$ with $\operatorname{div}^4(W)=0$ on $M$ is either Einstein or isometric to a finite quotient of of $N^{n-k}\times \mathbb{R}^k$, $(k > 0)$ the product of a Einstein manifold $N^{n-k}$ with the Gaussian shrinking soliton $\mathbb{R}^k$.
\end{theorem}

This theorem improves the results on gradient shrinking solitons with harmonic Weyl tensor in \cite{fergar, munses}. In the case of steady and expanding solitons, under natural Ricci curvature assumptions, we show that the soliton has harmonic Weyl curvature. Namely, we have the following theorems.

\begin{theorem}\label{thm_ste} Let $(M^n,g)$, $n\geq 4$, be a complete gradient steady Ricci soliton with positive Ricci curvature and such that the scalar curvature attains its maximum at some point. If $\operatorname{div}^4(W)=0$ on $M$ , then $(M^n,g)$ has harmonic Weyl curvature.
\end{theorem}

\begin{theorem}\label{thm_exp} Let $(M^n,g)$, $n\geq 4$, be a complete gradient expanding Ricci soliton with nonnegative Ricci curvature. If $\operatorname{div}^4(W)=0$ on $M$, then $(M^n,g)$ has harmonic Weyl curvature.
\end{theorem}

%

In dimension three, in the steady and expanding cases, we can prove stronger results. Namely we have the following theorems.

\begin{theorem}\label{thm_ste3d} Every three dimensional complete gradient steady Ricci soliton with $\operatorname{div}^3(C)=~0$ on $M$ is isometric to either a finite quotient of $\mathbb{R}^3$ or the Bryant soliton (up to scaling).
\end{theorem}

\begin{theorem}\label{thm_exp3d} Every three dimensional complete gradient expanding Ricci soliton with nonnegative Ricci curvature and $\operatorname{div}^3(C)=0$ on $M$ is rotationally symmetric.
\end{theorem}

Note that, in dimension three, the Bach tensor is defined as (see equation \eqref{def_Bach_comp})
$$
B_{ij}=\nabla_k C_{ijk}\,.
$$
Hence, in this case, the condition $\operatorname{div}^3(C)= 0$ is equivalent to $\operatorname{div}^2 (B)=\nabla_i \nabla_j B_{ij}= 0$ and Theorems \ref{thm_ste3d} and \ref{thm_exp3d} improve the results  \cite[Corollary 1.3]{caocatchemanmaz} and \cite[Theorem 5.9]{caocatchemanmaz}, respectively.

Note also that on the steady three dimensional gradient Ricci soliton $\Sigma^2 \times\RR$ the ``triple divergence'' of the Cotton tensor $\operatorname{div}^3(C)$ does not vanish identically, see Lemma \ref{lem_sig}.

Finally note that, as it will be clear from the proof, the scalar assumptions on the vanishing of $\operatorname{div}^3(C)$ and $\operatorname{div}^4(W)$ in all the above theorems can be trivially relaxed to a (suitable) inequality. For instance, Theorem \ref{thm_ste3d} holds just assuming $\operatorname{div}^3(C)\leq 0$ on $M$.

\

The proof of our results relies heavily on two main ingredients. The first is given by new integrability conditions for gradient Ricci solitons that we show in Proposition \ref{pro_int}. The second one is an integral formula (see Theorem \ref{thm_intide}) which relates the squared norm $|C|^2$ with the double divergence of the Cotton tensor $C$, for every gradient Ricci soliton and for a suitable family of cutoff functions (depending on the potential $f$) having compact support. A careful (double) use of this identity allows us to avoid imposing any Lebesgue integrability assumptions on the curvature of the soliton in our main results.

\section{Definitions and notation}
\label{sec_def}

The Riemann curvature
operator of a Riemannian manifold $(M^n,g)$ is defined 
as in~\cite{gahula} by
$$
\mathrm{Riem}(X,Y)Z=\nabla_{Y}\nabla_{X}Z-\nabla_{X}\nabla_{Y}Z+\nabla_{[X,Y]}Z\,.
$$ 
Throughout the article, the Einstein convention of summing over the repeated indices will be adopted. In a local coordinate system the components of the $(3,1)$-Riemann 
curvature tensor are given by
$R^{l}_{ijk}\tfrac{\partial}{\partial
  x^{l}}=\mathrm{Riem}\big(\tfrac{\partial}{\partial
  x^{i}},\tfrac{\partial}{\partial
  x^{j}}\big)\tfrac{\partial}{\partial x^{k}}$ and we denote by
$R_{ijkl}=g_{lm}R^{m}_{ijk}$ its $(4,0)$-version. The Ricci tensor is obtained by the contraction 
$R_{ik}=g^{jl}R_{ijkl}$ and $R=g^{ik}R_{ik}$ will 
denote the scalar curvature. The so called Weyl tensor is then 
defined by the following decomposition formula (see~\cite[Chapter~3,
Section~K]{gahula}) in dimension $n\geq 3$,
\begin{eqnarray}
\label{Weyl}
W_{ijkl}  & = & R_{ijkl} \, - \, \frac{1}{n-2} \, (R_{ik}g_{jl}-R_{il}g_{jk}
+R_{jl}g_{ik}-R_{jk}g_{il})  \nonumber \\
&&\,+\frac{R}{(n-1)(n-2)} \,
(g_{ik}g_{jl}-g_{il}g_{jk})\, \, .
\end{eqnarray}
The Weyl tensor shares the symmetries of the curvature
tensor. Moreover, as it can be easily seen by the formula above, all of its contractions with the metric are zero, i.e. $W$ is totally trace-free. In dimension three, $W$ is identically zero on every Riemannian manifold, whereas, when $n\geq 4$, the vanishing of the Weyl tensor is
a relevant condition, since it is  equivalent to the local
  conformal flatness of $(M^n,g)$. We also recall that in dimension $n=3$,  local conformal
  flatness is equivalent to the vanishing of the Cotton tensor
\begin{equation}\label{def_cot}
C_{ijk} =  R_{ij,k} - R_{ik,j}  - 
\frac{1}{2(n-1)}  \big( R_k  g_{ij} -  R_j
g_{ik} \big)\,,
\end{equation}
where $R_{ij,k}=\nabla_k R_{ij}$ and $R_k=\nabla_k R$ denote, respectively, the components of the covariant derivative of the Ricci tensor and of the differential of the scalar curvature.
By direct computation, we can see that the Cotton tensor $C$
satisfies the following symmetries
\begin{equation}\label{CottonSym}
C_{ijk}=-C_{ikj},\,\quad\quad C_{ijk}+C_{jki}+C_{kij}=0\,,
\end{equation}
moreover it is totally trace-free, 
\begin{equation}\label{CottonTraces}
g^{ij}C_{ijk}=g^{ik}C_{ijk}=g^{jk}C_{ijk}=0\,,
\end{equation}
by its skew--symmetry and Schur lemma.  Furthermore, it satisfies
\begin{equation}\label{eq_nulldivcotton}
C_{ijk,i} = 0,
\end{equation}
see for instance \cite[Equation 4.43]{catmasmonrig}. We recall that, for $n\geq 4$,  the Cotton tensor can also be defined as one of the possible divergences of the Weyl tensor:
 \begin{equation}\label{def_Cotton_comp_Weyl}
 C_{ijk}=\pa{\frac{n-2}{n-3}}W_{tikj, t}=-\pa{\frac{n-2}{n-3}}W_{tijk, t}.
 \end{equation}
 A computation shows that the two definitions coincide (see e.g. \cite{alimasrig}).

 In what follows a relevant role will be played by the \emph{Bach tensor}, first introduced in general relativity by Bach, \cite{bac}. By definition we have
 \begin{equation}\label{def_Bach_comp}
   B_{ij} = \frac{1}{n-3}W_{ikjl, lk} + \frac{1}{n-2}R_{kl}W_{ikjl} = \frac{1}{n-2}\pa{C_{jik, k}+R_{kl}W_{ikjl}}.
 \end{equation}

  A computation using the commutation rules for the second covariant derivative of the Weyl tensor or of the Schouten tensor (see \cite{catmasmonrig}) shows that the Bach tensor is symmetric (i.e. $B_{ij}=B_{ji}$); it is also evidently trace-free (i.e. $B_{ii}=0$). It is worth reporting here the following interesting formula for the divergence of the Bach tensor (see e. g. \cite{caoche2} for its proof)
\begin{equation}\label{diverBach}
  B_{ij, j} = \frac{n-4}{\pa{n-2}^2}R_{kt}C_{kti}.
\end{equation}

We recall here some useful equations satisfied by every gradient Ricci soliton $(M^n,g)$
 \begin{equation}\label{def_sol}
   R_{ij}+f_{ij}=\lambda g_{ij}, \quad \lambda \in \erre,
 \end{equation}
where $f_{ij}=\nabla_i\nabla_j f$ are the components of the Hessian of $f$ (see e.g. \cite{emilanman}).
\begin{lemma} Let $(M^n,g)$ be a gradient Ricci soliton of dimension $n\geq 3$. Then
\begin{equation}\label{eq_tra}
\Delta f + R = n \lambda
\end{equation}
\begin{equation}\label{eq_sch}
R_i = 2 f_t R_{it}
\end{equation}
\begin{equation}\label{eq_hamide}
R + |\nabla f|^2 = 2\lambda f + c
\end{equation}
for some constant $c\in\RR$.

\end{lemma}

The tensor $D$, introduced by H.-D. Cao and Q. Chen  in \cite{caoche1}, turns out to be a fundamental tool in the study of the geometry of gradient Ricci solitons (more in general for gradient Einstein-type manifolds, see \cite{catmasmonrig2}). In components it is defined as
 \begin{align}\label{def_D}
   D_{ijk}=&\frac{1}{n-2}\pa{f_kR_{ij}-f_jR_{ik}}+\frac{1}{(n-1)(n-2)}f_t\pa{R_{tk}g_{ij}-R_{tj}g_{ik}}\\\nonumber
 &\,-\frac{R}{(n-1)(n-2)}\pa{f_k g_{ij}-f_j g_{ik}}.
 \end{align}
 The $D$ tensor is skew-symmetric in the second and third indices (i.e. $D_{ijk}=-D_{ikj}$) and totally trace-free (i.e. $D_{iik}=D_{iki}=D_{kii}=0$).
Note that our convention for the tensor $D$ differs from that in \cite{caoche1}.

\section{Integrability conditions for gradient Ricci solitons}
\label{sec_int}

In this short section we present the four integrability conditions for gradient Ricci solitons of dimension $n\geq 3$. 
\begin{proposition}\label{pro_int}
  If $\pa{M^n, g}$ is a gradient Ricci soliton with potential function $f$, then the Cotton tensor,  the Bach tensor and the tensor $D$ satisfy the following conditions
\begin{eqnarray}
\label{eq_1int} C_{ijk}+f_t W_{tijk} &=& D_{ijk}, \\ \label{eq_2int} (n-2)B_{ij} -\pa{\frac{n-3}{n-2}}f_tC_{jit} &=& D_{ijk, k},\\
  \label{eq_3int}R_{kt}C_{kti}&= &(n-2)D_{itk, tk},\\\label{eq_4int}
\frac{1}{2}|C|^2+R_{kt}C_{kti, i}& = &(n-2)D_{itk, tki}.
\end{eqnarray}

\end{proposition}
\begin{proof} For the sake of completeness we will prove all the identities. First we prove \eqref{eq_1int} and \eqref{eq_2int}, which were first obtained in \cite{caoche2}. We start taking the covariant derivative of the soliton equation \eqref{def_sol}, which gives 
\begin{equation*}
  R_{ij, k}+ f_{ijk} =0.
\end{equation*}
Skew-symmetrizing, recalling the commutation relation $f_{ijk}-f_{ikj}=f_tR_{tijk}$ and using \eqref{Weyl} we deduce
\begin{eqnarray}\label{CI_A}
  R_{ij, k}-R_{ik, j} &=& f_tR_{tikj} \\ \nonumber
   &=& f_t\Big\{W_{ijkj} +\frac{1}{n-2}(R_{tk}g_{ij}-R_{tj}g_{ik}+R_{ij}g_{tk}-R_{ik}g_{tj})\\\nonumber
  &&\,-\frac{R}{(n-1)(n-2)}(g_{tk}g_{ij}-g_{tj}g_{ik}) \Big\}
 \\ \nonumber &=&-f_tW_{tijk} +\frac{1}{n-2}\pa{f_tR_{tk}g_{ij}-f_tR_{tj}g_{ik}+f_kR_{ij}-f_jR_{ik}}\\ \nonumber
&&\,-\frac{R}{(n-1)(n-2)}\pa{f_{k}g_{ij}-f_{j}g_{ik}}.
\end{eqnarray}
Now we insert in the previous equation the definitions of the Cotton tensor \eqref{def_cot} and of the tensor $D$, deducing \eqref{eq_1int}.

In order to prove \eqref{eq_2int} we take the divergence of \eqref{eq_1int}:
\begin{equation}\label{CI_B}
  C_{ijk, k}+f_{tk}W_{tijk}+f_tW_{tijk, k}= D_{ijk, k}.
\end{equation}
Inserting in \eqref{CI_B} the definition of the Bach tensor \eqref{def_Bach_comp}, the soliton equation \eqref{def_sol} and \eqref{def_Cotton_comp_Weyl}, and exploiting the fact that the Weyl tensor is totally trace free, equation \eqref{eq_2int} follows immediately.

In order to show \eqref{eq_3int} we take the covariant derivative of equation \eqref{eq_2int}, obtaining
  \[
  (n-2)B_{ij, k} -\pa{\frac{n-3}{n-2}}\pa{f_{tk}C_{jit}+f_tC_{jit, k}}= D_{ijt, tk}  ,
  \]
  which implies, using the soliton equation and the fact that $C_{jit}=-C_{jti}$,
  \[
    (n-2)B_{ij, k} -\pa{\frac{n-3}{n-2}}\pa{\lambda C_{jik}+ R_{tk}C_{jti} + f_tC_{jit, k}}= D_{ijt, tk} .
  \]
  Tracing with respect to $j$ and $k$, using equation \eqref{eq_nulldivcotton} and the fact that the Cotton tensor is totally trace-free we get
  \[
   (n-2)B_{ik, k} - \pa{\frac{n-3}{n-2}}R_{tk}C_{kti}= D_{ikt, tk} .
  \]
  Now we exploit \eqref{diverBach} and $D_{ijk}=-D_{ikj}$ in the previous relation, obtaining \eqref{eq_3int}.

Equation \eqref{eq_4int} follows by taking the divergence of \eqref{eq_3int},
\[
R_{kt, i}C_{kti}+R_{kt}C_{kti, i} = (n-2)D_{itk, tki}.
\]
Now we use the symmetry of the Cotton tensor, obtaining
\[
\frac{1}{2}\pa{R_{kt, i}-R_{ki, t}}C_{kti}+R_{kt}C_{kti, i}= (n-2)D_{itk, tki},
\]
from which we immediately deduce \eqref{eq_4int}.
\end{proof}

\section{A key integral formula}

In this section we show an integral formula that holds on every gradient Ricci soliton.
\begin{theorem}\label{thm_intide} Let $(M^n,g)$, $n\geq 3$, be a gradient Ricci soliton with potential function $f$. For every $\psi:\mathbb{R}\to\mathbb{R}$, $C^2$ function with $\psi(f)$ having compact support in $M$, one has
$$
\frac{1}{2}\int_M |C|^2 \,\psi(f) \, dV_g \,=\, - \int_M C_{kti,it} \, f_k \,\psi(f)\,dV_g \,.
$$
In particular, if $n\geq 4$, this is equivalent to
$$
\frac{1}{2}\int_M |C|^2 \,\psi(f) \, dV_g \,=\, \frac{n-2}{n-3} \int_M W_{lkti,lit} \, f_k \,\psi(f)\,dV_g \,.
$$
\end{theorem}

\begin{proof} Let $\psi$ satisfies the hypotheses. We multiply equation \eqref{eq_4int} by $\psi(f)$ and integrate over $M$. By using the soliton equation, the fact that $C$ is totally trace-free and integrating by parts, we obtain
\begin{eqnarray*}
\frac{1}{2}\int_M |C|^2 \,\psi(f) -\int_M C_{kti,i}\,f_{kt} \psi(f) &=& (n-2) \int_M D_{itk,tki} \,\psi(f) \\
&=& -(n-2)\int_M D_{itk,tk} \,f_i\,\psi'(f) \\
&=& -\int_M C_{kti} \,R_{kt}\,f_i \,\psi'(f) \\
&=& \int_M C_{kti} \,f_{kt}\,f_i \,\psi'(f) ,
\end{eqnarray*}
where we have exploited also equation \eqref{eq_3int}. Another integration by parts and the fact that $C_{kti}=-C_{kit}$ yield
\begin{eqnarray*}
\frac{1}{2}\int_M |C|^2 \,\psi(f) -\int_M C_{kti,i}\,f_{kt} \psi(f) &=& -\int_M C_{kti,t} \,f_{k}\,f_i \,\psi'(f) -\int_M C_{kti} \,f_{k}\,f_{it} \,\psi'(f)\\
&& -\int_M C_{kti} \,f_{k}\,f_i \,f_t\,\psi''(f) \\
&=& -\int_M C_{kti,t} \,f_{k}\,f_i \,\psi'(f).
\end{eqnarray*}
Hence, renaming indeces,
\begin{eqnarray*}
\frac{1}{2}\int_M |C|^2 \,\psi(f) -\int_M C_{kti,i}\,f_{kt} \psi(f) &=& -\int_M C_{kit,i} \,f_{k}\,f_t \,\psi'(f) \\
&=& \int_M C_{kti,i} \,f_{k}\,\big(\psi(f)\big)_t \\
&=& -\int_M C_{kti,it} \,f_{k}\,\psi(f) -\int_M C_{kti,i} \,f_{kt}\,\psi(f)\,.
\end{eqnarray*}
Simplifying we obtain the result. The second equation in the statement follows from \eqref{def_Cotton_comp_Weyl}.

\end{proof}

\begin{rem}
In case $n=3$ the formula of Theorem \ref{thm_intide} trivially holds. Indeed it is easy to see, using formulas \eqref{eq_1int} and \eqref{eq_3int}, that $$C_{ijk}=D_{ijk},\quad\quad\quad R_{kt}C_{kti}=C_{itk,tk},$$ hence, from the symmetries of $C$, one has
$$\frac{1}{2}|C|^2=\frac{1}{2}C_{kti}D_{kti}=\frac{1}{2}C_{kti}(R_{kt}f_i-R_{ik}f_{t})=C_{kti}R_{kt}f_i=C_{itk,tk}f_i=-C_{kti,it}f_k,$$ i.e. point wise on $M^3$
\begin{equation}\label{eq_poiide}
\frac{1}{2}|C|^2=-C_{kti,it}f_k.
\end{equation}
\end{rem}

\

\section{Proof of the results}

\subsection{Shrinking Ricci solitons}

In this section we prove Theorem \ref{thm_shr}. Let $(M^n,g)$, $n\geq4$, be a complete gradient shrinking Ricci soliton with potential function $f$.

If $M$ is compact, then we choose $\psi(f)\equiv1$ on $M$ in Theorem \ref{thm_intide}. Thus, integrating by parts, we have
$$\frac{1}{2}\int_M|C|^2=-\int_MC_{kti,it}f_k=\int_MC_{kti,itk}f.$$
By formula \eqref{def_Cotton_comp_Weyl} one has
\begin{equation}\label{eq127}
C_{kti,itk}=-\frac{n-2}{n-3}W_{jkti,jitk}=-\frac{n-2}{n-3}\operatorname{div}^4(W)\equiv0
\end{equation}
on $M$, by assumption. Hence, we conclude that $C\equiv0$ on $M$. Theorem \ref{thm_shr} now follows from \cite{fergar,munses}.

If $M$ is complete and noncompact, then we choose $\psi(f)=e^{-f}\phi(f)$, where, for any fixed $s>0$, $\phi:\RR\rightarrow\RR$ is a nonnegative $C^3$ function such that $\phi\equiv 1$ on $[0,s]$, $\phi\equiv 0$ on $[2s,+\infty)$ and $\phi'\leq 0$ on $[s,2s]$. It is well known that on every complete, noncompact gradient shrinking soliton the potential function $f$ is proper with quadratic growth at infinity (see \cite{caozho}). Then, for every $s>0$, the cutoff function $\psi(f)$ has compact support in $M$. Thus, integrating by parts in the integral formula of Theorem \ref{thm_intide} we obtain
\begin{eqnarray}\nonumber
\frac{1}{2}\int_M|C|^2 e^{-f}\phi(f)&=& -\int_M C_{kti,it}\,f_k\, e^{-f}\phi(f) \\\label{eq233}
&=& \int_M C_{kti,it}\big(e^{-f}\big)_k \,\phi(f)\\\nonumber
&=& -\int_M C_{kti,itk}\,e^{-f}\,\phi(f) -\int_M C_{kti,it}\,f_k\,e^{-f} \,\phi'(f).
\end{eqnarray}
Now, since the function $\widetilde{\psi}=e^{-f} \phi'(f)$ is $C^2$ with compact support, we can apply again the integral formula of Theorem \ref{thm_intide} with $\widetilde{\psi}$, obtaining
$$
-\int_M C_{kti,it}\,f_k\,e^{-f} \,\phi'(f) = -\int_M C_{kti,it}\,f_k\,\widetilde{\psi} = \frac{1}{2}\int_M |C|^2 \widetilde{\psi} = \frac{1}{2}\int_M |C|^2 e^{-f} \,\phi'(f) \leq 0,
$$
since $\phi'\leq 0$. Equation \eqref{eq233} yields
$$
\frac{1}{2}\int_M|C|^2 e^{-f}\phi(f) \leq -\int_M C_{kti,itk}\,e^{-f}\,\phi(f) = 0
$$
by assumption and equation \eqref{eq127}. Hence, $C\equiv 0$ on the compact set $\Omega_s=\{f\leq s\}$, since $\phi(f)\geq 0$ on $M$ and $\phi(f)\equiv 1$ on $\Omega_s$. Then $C\equiv 0$ on $M$, by taking the limit as $s \to +\infty$, and the conclusion follows again from \cite{fergar, munses}.

\subsection{Steady and expanding Ricci solitons in dimension greater than four} In this section we prove Theorem \ref{thm_ste} and Theorem \ref{thm_exp}. Let $(M^n,g)$, $n\geq4$, be a complete gradient steady or expanding Ricci soliton with potential function $f$.

It is well known that, if $M$ is compact, then $(M^n,g)$ is Einstein. In particular $(M^n,g)$ has harmonic Weyl curvature.

On the other hand, if $M$ is complete and noncompact, then we proceed similarly as in the shrinking case. We let $\psi(f)=e^{f}\phi(-f)$, where, for any fixed $s>0$, $\phi:\RR\rightarrow\RR$ is a nonnegative $C^3$ function such that $\phi\equiv 1$ on $[0,s]$, $\phi\equiv 0$ on $[2s,+\infty)$ and $\phi'\leq 0$ on $[s,2s]$. The assumptions of Theorem \ref{thm_ste} and Theorem \ref{thm_exp} imply that $-f$ is proper, with linear or quadratic growth in the steady or expanding case respectively (see for instance \cite{caoche1} and \cite{caocatchemanmaz}). Then, for every $s>0$, the cutoff function $\psi(f)$ has compact support in $M$. Integrating by parts in the integral formula of Theorem \ref{thm_intide} we obtain
\begin{eqnarray}\nonumber
\frac{1}{2}\int_M|C|^2 e^{f}\phi(-f)&=& -\int_M C_{kti,it}\,f_k\, e^{f}\phi(-f) \\\nonumber
&=& -\int_M C_{kti,it}\big(e^{f}\big)_k \,\phi(-f)\\\label{eq23}
&=& \int_M C_{kti,itk}\,e^{f}\,\phi(-f) -\int_M C_{kti,it}\,f_k\,e^{f} \,\phi'(-f).
\end{eqnarray}
Now, since the function $\widetilde{\psi}=e^{f} \phi'(-f)$ is $C^2$ with compact support, applying the integral formula of Theorem \ref{thm_intide} with $\widetilde{\psi}$, we get
$$
-\int_M C_{kti,it}\,f_k\,e^{f} \,\phi'(-f) = -\int_M C_{kti,it}\,f_k\,\widetilde{\psi} = \frac{1}{2}\int_M |C|^2 \widetilde{\psi} = \frac{1}{2}\int_M |C|^2 e^{f} \,\phi'(-f) \leq 0,
$$
since $\phi'\leq 0$. Equation \eqref{eq23} yields
\begin{equation}\label{eq128}
\frac{1}{2}\int_M|C|^2 e^{f}\phi(-f) \leq \int_M C_{kti,itk}\,e^{f}\,\phi(-f) = 0
\end{equation}
by assumption and equation \eqref{eq127}. Hence, arguing as before, $C\equiv 0$ on $M$, i.e. $(M,g)$ has harmonic Weyl curvature.

\subsection{Steady and expanding Ricci solitons in dimension three} In this section we prove Theorem \ref{thm_ste3d} and Theorem \ref{thm_exp3d}. First of all, in the expanding case, arguing exactly as before, we obtain equation \eqref{eq128}. This again implies $C\equiv 0$ on $M$, i.e. $(M,g)$ is locally conformally flat. The rotational symmetry now follows from \cite{caoche1}.

Finally, let $(M^3,g)$ be a three dimensional complete gradient steady Ricci soliton. By B.-L. Chen \cite{che} we have that $g$ must have nonnegative sectional curvature. By Hamilton's identity \eqref{eq_hamide}
$$R+\abs{\nabla f}^2=c,$$ for some $c\in\RR$, and thus $g$ has bounded curvature. From Hamilton's strong maximum principle (see e.g. \cite{choluni}) we deduce that $\pa{M^3, g}$ is:
\begin{itemize}
\item[(i)] flat, or
\item[(ii)] it has strictly positive sectional curvature, or
\item[(iii)] it splits as a product $\Sigma^2\times\erre$, where $\Sigma^2$ is the cigar steady soliton.
\end{itemize}
In case (i) the proof is complete. If (ii) holds, by \cite[Proposition 2.3]{denzhu} we know that $f$ has a unique critical point. Thus, by Hamilton's identity \eqref{eq_hamide}, the scalar curvature attains its maximum. In particular, from \cite{caoche1}, $-f$ is proper and has linear growth at infinity. By the same argument we used before, we conclude that $C\equiv 0$ on $M$, i.e. $(M,g)$ is locally conformally flat and the result follows from \cite{caoche1}.

Finally we now show that case (iii) cannot occur, by proving that the steady soliton $\Sigma^2\times\erre$ does not satisfy
$$
\operatorname{div}^3(C)\equiv 0.
$$
We recall that Hamilton's cigar steady soliton is defined as the complete Riemannian surface $(\Sigma^2,\tilde{g})$, where $\Sigma^2=\RR^2$,
$$
\tilde{g} = \frac{dx^2+dy^2}{1+x^2+y^2}
$$
and with potential function
$$
\tilde{f}(x,y) = -\log(1+x^2+y^2).
$$
On $(\Sigma^2\times\erre,g)$ we adopt global coordinates $s,x,y$ and hence the metric and the potential take the form
$$
g=ds^2+\frac{dx^2+dy^2}{1+x^2+y^2}
$$
and
$$
f(s,x,y) = -\log(1+x^2+y^2).
$$
In particular, the Ricci tensor is diagonal with
$$
R_{xx}=R_{yy}=\frac{1}{2}R>0, \quad\quad R_{ss}=0.
$$
Moreover, Hamilton's identity \eqref{eq_hamide} implies that at the origin $O=(0,0,0)$ one has
$$
\nabla f(O) = \nabla R(O) = 0.
$$
Now the conclusion of Theorem \ref{thm_ste3d} is a consequence of the following lemma.
\begin{lemma}\label{lem_sig} Using the above notation, one has
$$
\operatorname{div}^3(C) (O) = \frac{1}{8}R(O)^3 \neq 0.
$$
\end{lemma}
\begin{proof}
From equations \eqref{eq_1int} and \eqref{eq_4int} $C=D$ and
\begin{equation}\label{eq_189}
C_{kti,itk} = D_{kti,itk} = - D_{kit,itk} = -\frac{1}{2}|D|^2 - R_{ti}D_{tik,k} = -\frac{1}{2}|D|^2 - R_{ij}D_{ijk,k}
\end{equation}
Since $\nabla f(O)=0$, from the definition of $D$
$$
D_{ijk}=(f_kR_{ij}-f_jR_{ik})+\frac{1}{2}f_t(R_{tk}g_{ij}-R_{tj}g_{ik})-\frac{1}{2}R(f_k g_{ij}-f_j g_{ik}),
$$
we have $D(O)=0$. In order to compute at the origin $O$ the last term in \eqref{eq_189}, we take the divergence of $D$ and we obtain
\begin{eqnarray*}
D_{ijk,k} &=& \Delta f R_{ij} + f_k R_{ij,k} - f_{jk}R_{ik}- f_j R_{ik,k} + \frac{1}{2}\big(f_{tk} R_{tk}g_{ij} + f_t R_{tk,k}g_{ij}\\
&&-f_{ti}R_{tj}-f_t R_{tj,i}-f_k R_k g_{ij}-R\Delta f g_{ij} + R_i f_j + R f_{ij} \big).
\end{eqnarray*}
Computing at the origin $O$, we obtain
$$
D_{ijk,k} = \Delta f R_{ij} - f_{jk}R_{ik} + \frac{1}{2}\big(f_{tk} R_{tk}g_{ij} -f_{ti}R_{tj}-R\Delta f g_{ij} + R f_{ij} \big).
$$
Using the soliton equations \eqref{def_sol} and \eqref{eq_tra}, one has at the origin
$$
D_{ijk,k} = -\frac{3}{2}RR_{ij} + \frac{3}{2}R_{jk}R_{ik} - \frac{1}{2}|\ricc|^2 g_{ij} + \frac{1}{2}R^2 g_{ij}.
$$
Tracing with the Ricci tensor, we get at the origin
$$
R_{ij} D_{ijk,k} = -2R|\ricc|^2 + \frac{3}{2}R_{ij}R_{jk}R_{ik} + \frac{1}{2}R^3
$$
Now,
$$
|\ricc|^2 = \frac{1}{2}R^2, \quad\quad R_{ij}R_{jk}R_{ik}=\frac{1}{4}R^3,
$$
and we obtain
$$
R_{ij} D_{ijk,k} = -\frac{1}{8}R^3.
$$
From \eqref{eq_189} the conclusion follows.
\end{proof}

\

\bibliographystyle{abbrv}

\bibliography{biblio_vartadiv}
\end{document}